\documentclass[a4,psfig,english,11pt,epsf,portrait]{article}
\usepackage{amsmath,amsfonts,latexsym,amscd,amssymb}
\usepackage{epsfig}
\usepackage{amsmath}
\usepackage{amsthm}
\usepackage{amssymb}
\usepackage{enumerate}
\usepackage{color}
\usepackage{hyperref}
\hypersetup{hidelinks}
\definecolor{purple}{rgb}{0.65, 0, 1}
\definecolor{orange}{rgb}{1,.5,0}
\definecolor{brown}{rgb}{.9,.73,.26}

\newtheorem{theorem}{Theorem}[section]
\newtheorem{remark}{Remark}[section]
\newtheorem{proposition}[theorem]{Proposition}
\newtheorem{definition}{Definition}[section]
\newtheorem{lemma}[theorem]{Lemma}

\numberwithin{equation}{section}

\def\R{{\mathbb {R}}}

\def\V{{\frac{\mu}{1+t}}}

\usepackage{color}

\def\R{\mathbb{R}}

\title{Global existence for   wave equations with 
scale-invariant time-dependent damping
 and time derivative nonlinearity}

\author{
Ahmad Z. Fino\\
{\it \small 
College of Engineering and Technology, American University of the Middle East, Kuwait}\\
Mohamed  Ali  Hamza\\
{\it \small 
Imam Abdulrahman Bin Faisal University,
 Dammam, 34212, Saudi Arabia}\\
}

\date{}

\begin{document}
	\maketitle

\begin{abstract}
This paper addresses the Cauchy problem for wave equations with scale-invariant time-dependent damping and nonlinear time-derivative terms, modeled as
$$
\left\{
\begin{array}{ll}
 \displaystyle\partial_{t}^2u - \ \Delta u +\V\partial_tu= f(\partial_tu), &\quad  x\in \mathbb{R}^n, t>0, \\
u(x,0)= u_0(x), \quad \partial_tu(x,0)= u_1(x) &\quad x\in\mathbb{R}^n,
\end{array}
\right.
$$
where $f(\partial_tu)=|\partial_tu|^p $ or $|\partial_tu|^{p-1}\partial_tu$ with $p>1$ and $\mu>0$. We prove global existence of small data solutions in low dimensions $1\leq n\leq 3$ by using energy estimates in appropriate Sobolev spaces. Our primary contribution is an existence result for  $p>1+2/ \mu$,
 in the one-dimensional case,  when  $\mu \le 2$, which in conjunction with prior blow-up results from \cite{Our2}, establish that the critical exponent for small data solutions in one dimension is  $p_G(1,\mu)=1+2/{\mu}$, when  $\mu \le 2$. To the best of our knowledge, this is the first identification of the critical exponent range
  for the time-dependent damped wave equations with  scale-invariant and time-derivative nonlinearity.

\end{abstract}

\medskip





\section{Introduction}
\par\quad

In this work, we study the global (in time) existence of small data solutions to the Cauchy problem for  the  semilinear  wave equations with 
scale-invariant time-dependent damping
and  power-nonlinearity of derivative type
\begin{equation}
\label{NLW}
\left\{
\begin{array}{ll}
 \displaystyle\partial_{t}^2u - \ \Delta u +\V\partial_tu= f(\partial_tu), &\quad  x\in \mathbb{R}^n, t>0, \\
u(x,0)= u_0(x), \quad \partial_tu(x,0)= u_1(x) &\quad x\in\mathbb{R}^n,
\end{array}
\right.
\end{equation}
where $1\leq n\leq 3$, $p>1$,  $\mu>0$, $f(\partial_tu)=|\partial_tu|^p $ or $|\partial_tu|^{p-1}\partial_tu$, and where the  initial data 
$u_0,$ and $u_1$ are  chosen in the appropriate energy space, namely
 \begin{equation}\label{initialdata}
 (u_0,u_1)\in H^{m+2}(\mathbb{R}^n)\times H^{m+1}(\mathbb{R}^n),\qquad\hbox{with} \quad m=\left\{\begin{array}{ll}0&\quad \hbox{when}\,\,n=1,\\\\
  1&\quad \hbox{when}\,\,n=2,3.\\
  \end{array}
  \right.
  \end{equation}

The corresponding linear equation  to \eqref{NLW}  
\begin{align}\label{Eq_Linear_Wave_Scale}
	\begin{cases}
		\displaystyle{\partial_{t}^2u-\Delta u+\frac{\mu}{1+t}\partial_{t}u=0,}&x\in\mathbb {R}^n,\ t>0,\\
		(u,\partial_{t}u)(0,x)=(u_0,u_1)(x),&x\in\mathbb {R}^n,
	\end{cases}
\end{align}
is invariant under the following transformation:
$$\tilde{u}(x,t)=u(\lambda x, \lambda (1+t)-1), \ \lambda>0.$$
 According to the classification introduced by \cite{Wirth2004}, the above scaling justifies the designation of the \textit{scale-invariant} case for \eqref{NLW}. Indeed, the behavior of the solutions to \eqref{Eq_Linear_Wave_Scale} is determined by the parameter $\mu$ which provides a borderline between the \emph{non-effective dissipation} and the \emph{effective dissipation}. In fact, the non-effective damping means its solution somehow having the behavior of the free wave equation $\partial_t^2u-\Delta u=0$, and the effective damping stands for its solution somehow having the behavior of the corresponding parabolic equation $\frac{\mu}{1+t}\partial_tu-\Delta u=0$. \\

  By ignoring the damping term
 $\V \partial_tu$, the  problem \eqref{NLW}
is reduced to the classical semilinear wave equation,  namely
\begin{equation}
\label{P}
\left\{
\begin{array}{ll}
 \partial_{t}^2u - \ \Delta u = |\partial_t u|^p &\quad x\in\R^n,\,\,t>0, \\
  u(x,0)=u_0(x), \quad \partial_tu(x,0)=u_1(x) &\quad x\in\R^n,
\end{array}
\right.
\end{equation}
 for which we have the Glassey exponent. This case  is characterized by a critical power, denoted by $p_G$,  and  given by
\begin{equation}\label{Glassey}
p_G=p_G(n):=1+\frac{2}{n-1}.
\end{equation}
More precisely,  if $p \le p_G$ then there is no global solution for  \eqref{P},
for small initial data, non negative   and compactly  supported, 
  and for $p> p_G$ a global solution exists for small initial data; see e.g. \cite{Hidano1,Hidano2,John1,Rammaha,Sideris,Tzvetkov,Zhou1}.\\

In recent years, many papers have been devoted to the study of blow-up results and lifespan
estimates for the  solution of \eqref{NLW}.
 Indeed,   Lai and Takamura  showed in \cite{LT2} a blow up region for $p$, namely $p\in (1,p_G(n+2\mu)]$.
 Furthermore, an important refinement  was performed in
 \cite{Palmieri}, where the new bound, for $\mu \in (0,2)$. 
Note that the result in \cite{Palmieri} was later improved in \cite{Our2} by extending the upper bound for $p$ to $p_G(n+\mu)$, for any $\mu>0$.   In other words,  the upper bound is given by:
\begin{equation}\label{critical}
 p_G(n+\mu)=1+\frac{2}{n-1+\mu}, \quad \quad n\ge 1.
\end{equation}
It is natural  to investigate  the  critical power
 denoted here by $p_G(n,\mu)$, delimiting blow-up and global existence regions of small data solutions of problem \eqref{NLW}.\\

\par

%

To establish global existence results for the given problem, it is essential to analyze the associated linear problem and utilize decay rates in appropriate Sobolev spaces. That question was solved   in \cite{Wirth2004}  (see also \cite{AM,U79,Rei}).   
Indeed, by exploiting the partial Fourier transformation with respect to spatial variables  and   some basic properties of the explicit representations of solutions in terms of Bessel functions   in \cite{Wirth2004}     
provide  the following  estimate
 $$\|(\nabla u(t), \partial_tu(t))\|_{L^2}\le  C  (1+t)^{-\min (1,\frac{\mu}2)},$$ 
where the constant $C$  depends on  the  initial data.\\

Similarly,  the analogous  estimate for the  linear problem 
with space-dependent damping, 
namely
$\partial_{t}^2u - \ \Delta u +\frac{\mu}
 {\sqrt{1+|x|^2}}
\partial_tu= 0,$  was done by Ikehata–
Todorova–Yordanov in \cite{Ikehata}. Indeed,  they showed that, in the case where the initial data  are compactly supported, the  solution
 satisfies the same kind of  energy estimates, namely
$\|(\nabla u(t), \partial_tu(t))\|_{L^2}\le C
    (1+t)^{-\gamma},$
 for some  $\gamma=\gamma (n,\mu)>0$,
 and where the constant $C$ depends on the support of the initial data. \\

We would like to draw the attention of the reader to
the fact that a better understanding of the dependence of the constant related to the initial data is crucial for establishing a global existence results for the coresponding nonlinear problem. In this context,   in recent work in  \cite{S13}, the author successfully established a global existence result for the nonlinear problem with space-dependent damping in an exterior domain. This achievement was made possible by skillfully employing weighted energy estimates for the associated linear problem.\\

Up to our knowledge, there is no result to
the global existence of the solution of the Cauchy problem (\ref{NLW}).  
In the present paper, we {\bf aim} to provide the small data global existence of the mild solution for the Cauchy problem (\ref{NLW}). To do this, the first step consists of the study of the associated homogeneous problem, by the use of some energy estimates which are totally new based on some estimates invented by many researchers, e.g. \cite{ITY},  and the use  of their decay rates in suitable space (Sobolev space) as stated in Section \ref{sec2}. Thanks to the fixed point theorem,  we deal with nonlinearity.
Therefore, the main important result of this paper 
is the proof of the following: 
\begin{equation}\label{critical1}
p_G(n,\mu)\le 1+\frac{2}{\alpha}, \quad \textrm{if} \quad 1\le n\le 3,
\end{equation}
where 
\begin{equation}\label{alpha}
\alpha:=\min (\mu,2).
\end{equation}
The above result gives the predicted result in the one-dimensional case when  $\mu \le 2$. Namely, by combining the previous blow-up results  and \eqref{critical1}, we deduce that      
\begin{equation}\label{critical11}
p_G(1,\mu)= 1+\frac{2}{\mu}.
\end{equation}

At the end of this section we prepare notation and several definitions used throughout this paper. 
We denote by $C$ a positive constant, which may change from line to line.
$L^p = L^p(\mathbb{R}^n)$
stands for the usual Lebesgue space,
and
$H^{k} = H^{k}(\mathbb{R}^n)$
for $k \in \mathbb{Z}_{\ge 0}$
is the  Sobolev space defined by:
\begin{align*}
	H^{k}(\mathbb{R}^n) = \left\{ f \in L^2(\mathbb{R}^n) ;
			\| f \|_{H^{k}} = \sum_{\ell = 0}^k \|  \partial_x^{\ell} f \|_{L^2} < \infty \right\}.
\end{align*}

The paper is organized as follows: we start in Section \ref{sec2} by providing the energy estimates of mild/strong solutions of the corresponding homogeneous equation of \eqref{NLW}. We derive the main results (Theorems  \ref{globalexistence1}, \ref{globalexistence2}) in Section \ref{sec3} while Section \ref{sec4} is dedicated to the proof of these theorems.

\section{Homogeneous equation}\label{sec2}
In this section, we consider the following homogeneous problem
\begin{equation}\label{1}
\begin{cases}
\displaystyle \partial_{t}^2u-\Delta u +\frac{\mu}{1+t} \partial_{t}u =0, &x\in \mathbb{R}^n,\,t>0,\\
u(x,0)=  u_0(x),\, \partial_{t}u(x,0)=  u_1(x),&x\in \mathbb{R}^n,\\
\end{cases}
\end{equation}
where $n\geq1$, $\mu>0$. To begin with, we give the definition of a strong solution to \eqref{1}. 
\begin{definition}[Strong solution]${}$\\
Let $n\geq1$, and $(u_0,u_1)\in H^2(\mathbb{R}^n)\times H^1(\mathbb{R}^n)$. A function $u$ is said to be a strong solution to \eqref{1} if
	$$	u\in  \mathcal{C}\left([0,\infty),H^2(\mathbb{R}^n)\right)\cap \mathcal{C}^{1}\left([0,\infty),H^1(\mathbb{R}^n)\right)\cap \mathcal{C}^{2}\left([0,\infty),L^2(\mathbb{R}^n)\right)
,$$
	and $u$ has initial data $u(x,0)=u_0(x)$, $\partial_tu(x,0)=u_1(x)$ and satisfies
	the equation in \eqref{1} in the sense of $L^2(\mathbb{R}^n)$.
\end{definition}
\begin{theorem}[Theorem 2.27  in \cite{Ikawa}]\label{existencetheorem}${}$\\
	Let $n\geq1$. For each $(u_0,u_1)\in H^{m+2}(\mathbb{R}^n)\times H^{m+1}(\mathbb{R}^n)$, $m\in\mathbb{N}$, there exists a unique strong solution $u$ to \eqref{1} such that
	$$u\in\bigcap_{j=0}^{m+2}  \mathcal{C}^{m+2-j}\left([0,\infty),H^j(\mathbb{R}^n)\right).$$
	\end{theorem}
Let
\begin{equation}\label{alpha0}
\alpha_0:=
\left\{
\begin{array}{ll}
 \mu\  &\textrm{  if}\ \  \ \mu \in (0,2),\\
2-\eta \  &\textrm{  if}\ \  \ \mu =2,\\
 2 &\textrm{  if}\ \  \ \mu >2,
\end{array}
\right.
\end{equation} 
where $\eta \in (0,2)$. We need to distinguish three cases depending on the value of the parameter $\mu$. In each case, we introduce a set of energy functionals. 

{\bf First case:  $\mu \in (0,2)$.} 
Now, we   introduce  the following functionals:
\begin{eqnarray}
E_0(u(t),\partial_tu(t),t):&=&\frac12\displaystyle\int_{\R^n}\big (|\nabla
u(t)|^2+
(\partial_tu(t))^2 \big ) {\mathrm{d}}x,\label{En}\qquad
\\
E_1(u(t),\partial_tu(t),t):&=&\displaystyle\int_{\R^n}\big(u(t)\partial_tu(t)+ \frac{1}{2(t+1)}u^2(t)
 \big ) {\mathrm{d}}x,\quad \label{In}\\
 E_2(u(t),\partial_tu(t),t):&=&E_0(u(t),\partial_tu(t),t)+\frac{\mu  }{2(t+1)}E_1(u(t),\partial_tu(t),t).\quad\label{E3}
\end{eqnarray}
By evaluating the time derivative of $E_2(u(t),\partial_tu(t),t)$, we prove  the following
\begin{lemma} \label{LE} Let $\mu\in (0,2)$. Assume that $(u_0,u_1)\in H^{2}(\mathbb{R}^n)\times H^{1}(\mathbb{R}^n)$,  then the strong solution $u$ of \eqref{1} satisfies, 
for all   $t\ge s\ge0$,
\begin{equation} \label{A1}
	\frac{\|u(t)\|^2_{L^2}}{(1+t)^{2}}
+	\|\partial_tu(t)\|^2_{L^2}+\|\nabla u(t)\|^2_{L^2}\leq C
 \left(\frac{1+s}{1+t}\right)^{\mu}
\left( 
	\frac{\|u(s)\|^2_{L^2}}{(1+s)^{2}}
+	\|\partial_tu(s)\|^2_{L^2}+\|\nabla u(s)\|^2_{L^2}\right).
\end{equation}
\end{lemma}
\begin{proof}	
Assume that
$(u_0,u_1)\in H^{2}(\mathbb{R}^n)\times H^{1}(\mathbb{R}^n)$, then  thanks to Theorem \ref{existencetheorem} the solution $u$ of \eqref{1} satisfies 
$u\in \bigcap_{j=0}^{2}  \mathcal{C}^{j}\left([0,\infty),H^{2-j}(\mathbb{R}^n)\right).$
Therefore,
 $E_0(u(t),\partial_tu(t),t)$ is a differentiable  function in time. Moreover, by multiplying \eqref{1} by $\partial_tu(t)$, and using the integration by parts formula, we get  for all   $t>0$,
\begin{equation}\label{E0}
 \frac{d}{dt}E_0(u(t),\partial_tu(t),t)=-\frac{\mu}{1+t}\displaystyle\int_{\R^n}
(\partial_tu(t))^2  {\mathrm{d}}x.
\end{equation}
Similarly, 
  $E_1(u(t),\partial_tu(t),t)$ is a differentiable  function. In a similar way, by multiplying \eqref{1} by $u(t)$, and  integration by parts, we get  for all   $t>0$,
\begin{eqnarray}\label{E1}
 \frac{d}{dt}E_1(u(t),\partial_tu(t),t)&=&\displaystyle\int_{\R^n}
(\partial_tu(t))^2  {\mathrm{d}}x
-\displaystyle\int_{\R^n}
|\nabla u(t)|^2  {\mathrm{d}}x\\
&&-\frac{1}{2(1+t)^2}\displaystyle\int_{\R^n}
(u(t))^2  {\mathrm{d}}x+\frac{1-\mu}{1+t}\displaystyle\int_{\R^n}
u(t)\partial_tu(t)  {\mathrm{d}}x.\nonumber
\end{eqnarray}
Therefore, exploiting  \eqref{E0}, \eqref{E1} and  the definition of $E_2(u(t),\partial_tu(t),t)$ given by \eqref{E3}, we  infer
\begin{eqnarray}\label{E11}
 \frac{d}{dt}E_2(u(t),\partial_tu(t),t)&=&-\frac{\mu}{2(1+t)} \displaystyle\int_{\R^n}
(\partial_tu(t))^2  {\mathrm{d}}x
-\frac{\mu}{2(1+t)}\displaystyle\int_{\R^n}
|\nabla u(t)|^2  {\mathrm{d}}x\nonumber\\
&&-\frac{\mu}{4(1+t)^3}\displaystyle\int_{\R^n}
(u(t))^2  {\mathrm{d}}x+\frac{\mu-\mu^2}{2(1+t)^2}\displaystyle\int_{\R^n}
u(t)\partial_tu(t)  {\mathrm{d}}x\nonumber\\
&&-\frac{\mu}{2(1+t)^2 }\Big(\displaystyle\int_{\R^n}
u(t)\partial_tu(t)  {\mathrm{d}}x+\frac1{2(1+t)}
\displaystyle\int_{\R^n}
(u(t))^2  {\mathrm{d}}x\Big).
\end{eqnarray}
From \eqref{E11} and by using the definition of $E_2(u(t),\partial_tu(t),t)$, we have
\begin{eqnarray}\label{E12}
 \frac{d}{dt}E_2(u(t),\partial_tu(t),t)+\frac{\mu}{1+t}E_2(u(t),\partial_tu(t),t)
+\frac{\mu(2-\mu)}{4(1+t)^3}\displaystyle\int_{\R^n}
(u(t))^2  {\mathrm{d}}x=0.
\end{eqnarray}
Multiplying \eqref{E12} by  $(1+t)^{\mu}$ and integrating over $[s, t]$, we deduce that, for all $t\ge s\ge0$,
 \begin{equation}\label{E14}
 (1+t)^{\mu} E_2(u(t),\partial_tu(t),t)+\frac{\mu (2-\mu)}{4}\int_{s}^{t}\frac{1}{(1+\tau)^{3-\mu}}\displaystyle\int_{\R^n}
(u(\tau))^2  {\mathrm{d}}x{\mathrm{d}}\tau=
 (1+s)^{\mu} E_2(u(s),\partial_tu(s),s).
\end{equation}
In addition, by using Young's inequality 
$$|ab|\leq \varepsilon a^2+\frac{b^2}{4\varepsilon}\qquad\hbox{with}\quad a=\frac{u}{1+t},\,\, b=u_t,\,\, \frac{\mu}{4}<\varepsilon < \frac{1}{2},$$ 
and the fact that  $\mu\in (0,2)$, there exists $C=C(\mu )$ such that  
\begin{equation} \label{E15}C^{-1} E_2(u(t),\partial_tu(t),t)\le
\displaystyle\int_{\R^n}\left(
\frac{(u(t))^2}{(t+1)^2}+
|\nabla
u(t)|^2+
(\partial_tu(t))^2
 \right) {\mathrm{d}}x\le C E_2(u(t),\partial_tu(t),t), \quad t \ge 0.
\end{equation}
Finally, using \eqref{E14}, and \eqref{E15},  we deduce that
for all $ t \ge s\ge 0$,
\begin{equation} \label{E16}
\displaystyle\int_{\R^n}\left(\frac{(u(t))^2}{(t+1)^2}+|\nabla
u(t)|^2+
(\partial_tu(t))^2
 \right) {\mathrm{d}}x\le C 
 \left(\frac{1+s}{1+t}\right)^{\mu}\displaystyle\int_{\R^n}\left(
 \frac{(u(s))^2}{(s+1)^2}
 + |\nabla
u(s)|^2+
(\partial_tu(s))^2 \right) {\mathrm{d}}x.
\end{equation}
This ends the proof of Lemma  \ref{LE}.
\end{proof}	
%

%

{\bf Second case:  $\mu >2$.} 
Now, we   introduce  the following functionals:
\begin{eqnarray}
E_3(u(t),\partial_tu(t),t):&=&\displaystyle\int_{\R^n}\big( u(t)\partial_tu(t)+ \frac{\mu-1}{2(t+1)}u^2(t)
 \big ) {\mathrm{d}}x,\quad \label{Inbis}\\
 E_4(u(t),\partial_tu(t),t):&=&E_0(u(t),\partial_tu(t),t)+\frac{1}{t+1}E_3(u(t),\partial_tu(t),t).\quad\label{E3bis}
\end{eqnarray}

\begin{lemma} \label{LEbis} Let $\mu >2$. 
Assume that $(u_0,u_1)\in H^{2}(\mathbb{R}^n)\times H^{1}(\mathbb{R}^n)$,  then the strong solution $u$ of \eqref{1} satisfies, for all   $t\ge s\ge0$,
\begin{equation} \label{A1bis}
\frac{\|u(t)\|^2_{L^2}}{(1+t)^{2}}
+	\|\partial_tu(t)\|^2_{L^2}+\|\nabla u(t)\|^2_{L^2}\leq C
 \left(\frac{1+s}{1+t}\right)^{2}
\left( \frac{\|u(s)\|^2_{L^2}}{(1+s)^{2}}
+	\|\partial_tu(s)\|^2_{L^2}+\|\nabla u(s)\|^2_{L^2}\right).
\end{equation}
\end{lemma}

\begin{proof}	
By using a similar way as in Lemma \ref{LE}, we show that   $E_3(u(t),\partial_tu(t))$ is a differentiable function. Moreover, 
 by exploiting \eqref{1} and  integration by parts, we get  for all   $t>0$,
\begin{eqnarray}\label{E1bis}
 \frac{d}{dt}E_3(u(t),\partial_tu(t),t)&=&\displaystyle\int_{\R^n}
(\partial_tu(t))^2  {\mathrm{d}}x
-\displaystyle\int_{\R^n}
|\nabla u(t)|^2  {\mathrm{d}}x\\
&&-\frac{\mu-1}{2(1+t)^2}\displaystyle\int_{\R^n}
(u(t))^2  {\mathrm{d}}x
 -\frac1{t+1}\displaystyle\int_{\R^n} u(t)\partial_tu(t) {\mathrm{d}}x.\nonumber
\end{eqnarray}
Therefore, exploiting  \eqref{E0}, \eqref{E1bis} and  the definition of $E_4(u(t),\partial_tu(t),t)$ given by \eqref{E3bis}, we  infer
\begin{eqnarray}\label{E11bis}
 \frac{d}{dt}E_4(u(t),\partial_tu(t),t)&=&-\frac{\mu-1}{1+t} \displaystyle\int_{\R^n}
(\partial_tu(t))^2  {\mathrm{d}}x
-\frac{1}{1+t}\displaystyle\int_{\R^n}
|\nabla u(t)|^2  {\mathrm{d}}x\\
&&-\frac{\mu-1}{2(1+t)^3}\displaystyle\int_{\R^n}
(u(t))^2  {\mathrm{d}}x
 -\frac1{(t+1)^2}\displaystyle\int_{\R^n} u(t)\partial_tu(t) {\mathrm{d}}x
\nonumber\\
&&-\frac{1}{(1+t)^2 }\Big(\displaystyle\int_{\R^n}
u(t)\partial_tu(t)  {\mathrm{d}}x+\frac{\mu-1}{2(1+t)}
\displaystyle\int_{\R^n}
(u(t))^2  {\mathrm{d}}x\Big).\nonumber
\end{eqnarray}
From \eqref{E11bis} and \eqref{E3bis}, we have
\begin{eqnarray}\label{E12bis}
 \frac{d}{dt}E_4(u(t),\partial_tu(t),t)+\frac{2}{1+t}E_4(u(t),\partial_tu(t),t)
+\frac{\mu-2}{1+t}\displaystyle\int_{\R^n}
(\partial_tu(t))^2  {\mathrm{d}}x=0.
\end{eqnarray}
Multiplying \eqref{E12bis} by $(1+t)^{2}$, and integrating over $[s, t]$, we deduce that, for all $t\ge s\ge0$,
 \begin{equation}\label{E14bis}
 (1+t)^{2} E_4(u(t),\partial_tu(t),t)+(\mu -2)\int_{s}^{t}(1+\tau)\displaystyle\int_{\R^n}
(\partial_tu(\tau))^2  {\mathrm{d}}x{\mathrm{d}}\tau=
 (1+s)^{2} E_4(u(s),\partial_tu(s),s).
\end{equation}
Since $\mu>2$, using Young's inequality 
$$|ab|\leq \varepsilon a^2+\frac{b^2}{4\varepsilon}\qquad\hbox{with}\quad a=\frac{u}{1+t},\,\,b=u_t,\,\, \frac{1}{2}<\varepsilon<\frac{\mu-1}{2},$$ 
there exists $C=C(\mu )$ such that  
\begin{equation} \label{E15bis}C^{-1} E_4(u(t),\partial_tu(t),t)\le
\displaystyle\int_{\R^n}\left( \frac{(u(t))^2}{(t+1)^2}+
|\nabla
u(t)|^2+
(\partial_tu(t))^2
 \right) {\mathrm{d}}x\le C E_4(u(t),\partial_tu(t),t), \quad t \ge 0.
\end{equation}
Then, by using \eqref{E14bis}, and \eqref{E15bis},  we deduce that
for all $ t \ge s\ge 0$,
\begin{equation} \label{E16bis}
\displaystyle\int_{\R^n}\left( \frac{(u(t))^2}{(t+1)^2}+
|\nabla
u(t)|^2+
(\partial_tu(t))^2
 \right) {\mathrm{d}}x\le C 
 \left(\frac{1+s}{1+t}\right)^{2}\displaystyle\int_{\R^n}\left( \frac{(u(s))^2}{(s+1)^2}+
 |\nabla
u(s)|^2+
(\partial_tu(s))^2
 \right) {\mathrm{d}}x.
\end{equation}
This ends the proof of Lemma  \ref{LEbis}.
\end{proof}	

{\bf Third case:  $\mu=2$.} 
Let $\eta\in(0,2)$. Now, we   introduce  the following functionals:
\begin{eqnarray}
E_5(u(t),\partial_tu(t),t):&=&\displaystyle\int_{\R^n}\big( u(t)\partial_tu(t)+ \frac{1+\eta}{2(t+1)}u^2(t)
 \big ) {\mathrm{d}}x,\quad \label{Inbbis}\\
 E_6(u(t),\partial_tu(t),t):&=&E_0(u(t),\partial_tu(t),t)+\frac{2-\eta}{2(t+1)}E_5(u(t),\partial_tu(t),t).\quad\label{E3bbis}
\end{eqnarray}

\begin{lemma} \label{LEbbis} Let $\mu=2$ and $\eta\in(0,2)$. Assume that $(u_0,u_1)\in H^{2}(\mathbb{R}^n)\times H^{1}(\mathbb{R}^n)$,  then the strong solution $u$ of \eqref{1} satisfies, 
for all   $t\ge s\ge0$, 
\begin{equation} \label{A1bbis}
	\frac{\|u(t)\|^2_{L^2}}{(1+t)^{2}}
+	\|\partial_tu(t)\|^2_{L^2}+\|\nabla u(t)\|^2_{L^2}\leq C
 \left(\frac{1+s}{1+t}\right)^{2-\eta}
\left( 
	\frac{\|u(s)\|^2_{L^2}}{(1+s)^{2}}
+	\|\partial_tu(s)\|^2_{L^2}+\|\nabla u(s)\|^2_{L^2}\right).
\end{equation}
\end{lemma}
\begin{proof}	
Assume that
$(u_0,u_1)\in H^{2}(\mathbb{R}^n)\times H^{1}(\mathbb{R}^n)$, then  thanks to Theorem \ref{existencetheorem} the solution $u$ of \eqref{1} satisfies 
$u\in \bigcap_{j=0}^{2}  \mathcal{C}^{j}\left([0,\infty),H^{2-j}(\mathbb{R}^n)\right).$
Therefore,
 $E_0(u(t),\partial_tu(t),t)$ is a differentiable  function in time. Moreover, by multiplying \eqref{1} by $\partial_tu(t)$, and  integration by parts, we get  for all   $t>0$,
\begin{equation}\label{E0bbis}
 \frac{d}{dt}E_0(u(t),\partial_tu(t),t)=-\frac{2}{1+t}\displaystyle\int_{\R^n}
(\partial_tu(t))^2  {\mathrm{d}}x.
\end{equation}
Similarly, 
  $E_5(u(t),\partial_tu(t),t)$ is a differentiable  function. In a similar way, by multiplying \eqref{1} by $u(t)$, and  integration by parts, we get  for all   $t>0$,
\begin{eqnarray}\label{E1bbis}
 \frac{d}{dt}E_5(u(t),\partial_tu(t),t)&=&\displaystyle\int_{\R^n}
(\partial_tu(t))^2  {\mathrm{d}}x
-\displaystyle\int_{\R^n}
|\nabla u(t)|^2  {\mathrm{d}}x\\
&&-\frac{1+\eta}{2(1+t)^2}\displaystyle\int_{\R^n}
(u(t))^2  {\mathrm{d}}x-\frac{1-\eta}{1+t}\displaystyle\int_{\R^n}
u(t)\partial_tu(t)  {\mathrm{d}}x.\nonumber
\end{eqnarray}
Therefore, exploiting  \eqref{E0bbis}, \eqref{E1bbis} and  the definition of $E_6(u(t),\partial_tu(t),t)$ given by \eqref{E3bbis}, we  infer
\begin{eqnarray}\label{E11bbis}
 \frac{d}{dt}E_6(u(t),\partial_tu(t),t)&=&-\frac{2+\eta}{2(1+t)} \displaystyle\int_{\R^n}
(\partial_tu(t))^2  {\mathrm{d}}x
-\frac{2-\eta}{2(1+t)}\displaystyle\int_{\R^n}
|\nabla u(t)|^2  {\mathrm{d}}x\\
&&-\frac{(2-\eta)(1+\eta)}{4(1+t)^3}\displaystyle\int_{\R^n}
(u(t))^2  {\mathrm{d}}x-\frac{(2-\eta)(1-\eta)}{2(1+t)^2}\displaystyle\int_{\R^n}
u(t)\partial_tu(t)  {\mathrm{d}}x\nonumber\\
&&-\frac{2-\eta}{2(1+t)^2 }\Big(\displaystyle\int_{\R^n}
u(t)\partial_tu(t)  {\mathrm{d}}x+\frac{1+\eta}{2(1+t)}
\displaystyle\int_{\R^n}
(u(t))^2  {\mathrm{d}}x\Big).\nonumber
\end{eqnarray}
By \eqref{E11bbis} and using the definition of $E_6(u(t),\partial_tu(t),t)$, we have
\begin{eqnarray}\label{E12bbis}
&& \frac{d}{dt}E_6(u(t),\partial_tu(t),t)+\frac{2-\eta}{1+t}E_6(u(t),\partial_tu(t),t)\nonumber\\
&&\,\,+\frac{\eta(1+\eta)(2-\eta)}{4(1+t)^3}\displaystyle\int_{\R^n}
(u(t))^2  {\mathrm{d}}x+\frac{\eta}{1+t}\displaystyle\int_{\R^n}
(\partial_tu(t))^2  {\mathrm{d}}x=0.
\end{eqnarray}
Multiplying \eqref{E12bbis} by  $(1+t)^{2-\eta}$, and integrating over $[s, t]$, we deduce that, for all $t\ge s\ge0$,
\begin{eqnarray}\label{E14bbis}
 (1+t)^{2-\eta} E_6(u(t),\partial_tu(t),t)+\frac{\eta(1+\eta)(2-\eta)}{4}\int_{s}^{t}\frac{1}{(1+\tau)^{3-\mu}}\displaystyle\int_{\R^n}
(u(\tau))^2  {\mathrm{d}}x{\mathrm{d}}\tau\nonumber\\
+\eta\int_{s}^{t}(1+\tau)^{1-\eta}\displaystyle\int_{\R^n}
(\partial_\tau u(\tau))^2  {\mathrm{d}}x{\mathrm{d}}\tau=
 (1+s)^{2-\eta} E_6(u(s),\partial_tu(s),s).
\end{eqnarray}
In addition, by using Young's inequality 
$$|ab|\leq \frac{a^2}{2}+\frac{b^2}{2}\qquad \hbox{with}\quad a=\frac{u}{1+t},\,\, b=u_t,$$
and the fact that $\eta\in (0,2)$, there exists $C=C(\eta )$ such that  
\begin{equation} \label{E15bbis}C^{-1} E_6(u(t),\partial_tu(t),t)\le
\displaystyle\int_{\R^n}\left(
\frac{(u(t))^2}{(t+1)^2}+
|\nabla
u(t)|^2+
(\partial_tu(t))^2
 \right) {\mathrm{d}}x\le C E_6(u(t),\partial_tu(t),t), \quad t \ge 0.
\end{equation}
Finally, using \eqref{E14bbis}, and \eqref{E15bbis},  we deduce that
for all $ t \ge s\ge 0$,
\begin{equation} \label{E16bbis}
\displaystyle\int_{\R^n}\left(\frac{(u(t))^2}{(t+1)^2}+|\nabla
u(t)|^2+
(\partial_tu(t))^2
 \right) {\mathrm{d}}x\le C 
 \left(\frac{1+s}{1+t}\right)^{2-\eta}\displaystyle\int_{\R^n}\left(
 \frac{(u(s))^2}{(s+1)^2}
 + |\nabla
u(s)|^2+
(\partial_tu(s))^2 \right) {\mathrm{d}}x.
\end{equation}
This ends the proof of Lemma  \ref{LEbbis}.
\end{proof}		
Let us denote by $R(t,s)$ the operator which maps the initial data $(u(s),u_t(s))\in
H^{m+2}(\mathbb{R}^n)\times H^{m+1}(\mathbb{R}^n)$ given at the time $s\geq0$ to the solution $u(t)\in
H^{m+2}(\mathbb{R}^n)$ at the time $t \geq s$, i.e. the solution $u$ of (\ref{1})
is defined by $u(t)=R(t,0)(u_0,u_1)$.  We also write $S(t,s)g=R(t,s)(0,g)$ for a function $g\in H^{m+1}(\mathbb{R}^n)$. Finally, if $(u_0,u_1)\in H^{1}(\mathbb{R}^n)\times L^{2}(\mathbb{R}^n)$, the function $t\rightarrow R(t)(u_0,u_1)$, which is defined in the sense of \eqref{Extend.Operator1} below, is called a``generalized
solution" of the initial value problem \eqref{1} which we will call a {\bf mild} solution.

\begin{proposition}\label{prop2.1}
	Let $n\geq 1$. If $(u_0,u_1)\in H^{1}(\mathbb{R}^n)\times L^{2}(\mathbb{R}^n)$, then there exists a unique {\bf mild} solution 
	$$	u\in  \mathcal{C}\left([0,\infty),H^1(\mathbb{R}^n)\right)\cap \mathcal{C}^{1}\left([0,\infty),L^2(\mathbb{R}^n)\right),$$
	 of \eqref{1} satisfies
\begin{equation} \label{0A1}
\frac{\|u(t)\|^2_{L^2}}{
	(1+t)^{2}}
+	\|\partial_tu(t)\|^2_{L^2}+\|\nabla u(t)\|^2_{L^2}\leq C
 \left(\frac{1+s}{1+t}\right)^{\alpha_0}
\left( \frac{\|u(s)\|^2_{L^2}}{
	(1+s)^{2}}
+	\|\partial_tu(s)\|^2_{L^2}+\|\nabla u(s)\|^2_{L^2}\right),
\end{equation}
for  all  
  $t\ge s\ge0$, where $\alpha_0$ is defined in \eqref{alpha0}. In addition,  if $(u_0,u_1)\in H^{2}(\mathbb{R}^n)\times H^{1}(\mathbb{R}^n)$, then $u$ is a {\bf strong} solution and satisfies
\begin{equation}\label{5mai1}
\|\nabla \partial_tu(t)\|^2_{L^2}+\|\Delta u(t)\|^2_{L^2}\leq C
 \left(\frac{1+s}{1+t}\right)^{\alpha_0}
\left( 	\|\nabla \partial_tu(s)\|^2_{L^2}+\|\nabla u(s)\|^2_{H^1}\right),
\end{equation}
for  all  
  $t\ge s\ge0$. Furthermore, if $(u_0,u_1)\in H^{3}(\mathbb{R}^n)\times H^{2}(\mathbb{R}^n)$, 
 then $u$ satisfies
\begin{equation}\label{5mai5}
	\|\partial_t\Delta u(t)\|^2_{L^2}+\|\nabla \Delta u(t)\|^2_{L^2}\leq C
 \left(\frac{1+s}{1+t}\right)^{\alpha_0}
\left( 	\|\partial_t\Delta u(s)\|^2_{L^2}+\| \Delta u(s)\|^2_{H^1}\right),
\end{equation}
 for  all  
  $t\ge s\ge0$.
\end{proposition}	
\begin{proof}
{\bf Existence and Uniqueness.} Let $T_0>0$ an arbitrary number, and let $(u_0,u_1)\in H^{1}(\mathbb{R}^n)\times L^{2}(\mathbb{R}^n)$. By the density argument, there exist sequences 
$$\left\{(u_0^{(j)},u_1^{(j)})\right\}_{j=1}^\infty\subseteq
H^2(\mathbb{R}^n)\times H^1(\mathbb{R}^n),$$
such that
$$\lim_{j\rightarrow\infty}(u_0^{(j)},u_1^{(j)})=(u_0,u_1)\;\;\mbox{in}\;\; H^1(\mathbb{R}^n)\times L^2(\mathbb{R}^n).$$
Using Theorem \ref{existencetheorem}, let $u^{(j)}$ be the strong solution of the linear homogeneous equation \eqref{1} with the initial data $(u_0^{(j)},u_1^{(j)})$. Then, the difference $u^{(j)}-u^{(k)}$ is a strong solution of the Cauchy problem 
$$
\begin{cases}
\displaystyle \partial_{t}^2u-\Delta u +\frac{\mu}{1+t} \partial_{t}u =0, &x\in \mathbb{R}^n,\,t>0,\\\\
u(x,0)= u^{(j)}_0(x)-u^{(k)}_0(x),\, \partial_{t}u(x,0)=  u_1^{(j)}(x)-u_1^{(k)}(x),&x\in \mathbb{R}^n.\\
\end{cases}
$$
Apply Lemmas \ref{LE}, \ref{LEbis}, and \ref{LEbbis} to $u^{(j)}-u^{(k)}$, we have in particular
\begin{eqnarray*}
&&\|u^{(j)}-u^{(k)}\|^2_{L^2}+\|\partial_t(u^{(j)}-u^{(k)})\|^2_{L^2}+\|\nabla (u^{(j)}-u^{(k)})\|^2_{L^2}\\
&&\leq C (1+T_0)^{2-\alpha_0}\left( \|u_0^{(j)}-u_0^{(k)}\|^2_{L^2}+\|u_1^{(j)}-u_1^{(k)}\|^2_{L^2}+\|\nabla (u_0^{(j)}-u_0^{(k)})\|^2_{L^2}\right),
\end{eqnarray*}
for  all  $t\in [0,T_0]$, where $\alpha_0$ is defined in \eqref{alpha0}.  This shows that $\left\{u^{(j)}\right\}_{j=1}^\infty$ is a Cauchy sequence in the complete space $C([0,T_0];H^1(\mathbb{R}^n))\cap C^1([0,T_0];L^2(\mathbb{R}^n))$. Therefore, we can define the limit 
\begin{equation}\label{eq2.9}
\lim_{j\rightarrow\infty}u^{(j)}=u\in C([0,\infty);H^1(\mathbb{R}^n))\cap C^1([0,\infty);L^2(\mathbb{R}^n)),
\end{equation}
since $T_0>0$ is arbitrary. As $u^{(j)}$ satisfies $u^{(j)}(t,x)=R(t)(u^{(j)}_0,u^{(j)}_1)$, then
$$
u(t)=\lim_{j\rightarrow\infty} R(t)(u^j_0,u^j_1),
$$
this means that the operator $R(t)$ has been extended uniquely to a new operator
	\begin{align}\label{Extend.Operator1}
	\widetilde{R}(t)&:\,H^1(\mathbb{R}^n)\times L^2(\mathbb{R}^n)\,\longrightarrow\, X_0\\
	&\qquad\quad(u_0,u_1)\qquad\longmapsto u(t)\notag
	\end{align}
	also denoted by $R(t)$, where
	$$X_0:=\mathcal{C}\left([0,\infty),H^1(\mathbb{R}^n)\right)\cap
	\mathcal{C}^1\left([0,\infty),L^2(\mathbb{R}^n)\right),$$
which indicates that $u(t)=R(t)(u_0,u_1)$ is a unique mild solution of \eqref{1}.\\ 
{\bf Energy estimate \eqref{0A1}.} By Lemmas \ref{LE}, \ref{LEbis}, and \ref{LEbbis}, each strong solution $u^{(j)}$ constructed above satisfies the estimates \eqref{A1}, \eqref{A1bis}, \eqref{A1bbis}. By letting $j\rightarrow\infty$ and using (\ref{eq2.9}), the same estimates hold for the mild solution $u$, and we deduce that estimate \eqref{0A1} holds.\\
{\bf Energy estimate \eqref{5mai1}.} Let $1\le i\le n$. A straightforward computation implies that $U_i=\partial_i u$, satisfies
\begin{equation}\label{ut}
\begin{cases}
\displaystyle\partial_{t}^2U_i-\Delta U_i +\frac{\mu}{1+t} \partial_{t}U_i=0, &x\in \mathbb{R}^n,\,t>0,\\
U_i(x,0)=  \partial_iu_0(x),\,\partial_{t}U_i(x,0)=  \partial_iu_1(x),& x\in \mathbb{R}^n.\\
\end{cases}
\end{equation}
Note that 
$(u_0,u_1)\in H^{2}(\mathbb{R}^n)\times H^{1}(\mathbb{R}^n)$, then
$(\partial_iu_0,\partial_iu_1)\in H^{1}(\mathbb{R}^n)\times L^{2}(\mathbb{R}^n)$. By \eqref{0A1}, the mild solution $U_i$ of \eqref{ut} satisfies, for  all $t\ge s\ge0$,
\begin{equation} \label{5mai2}
\frac{\|U_i(t)\|^2_{L^2}}{
	(1+t)^{2}}
+	\|\partial_t U_i(t)\|^2_{L^2}+\|\nabla U_i(t)\|^2_{L^2}\leq C
 \left(\frac{1+s}{1+t}\right)^{\alpha_0}
\left( \frac{\|U_i(s)\|^2_{L^2}}{
	(1+s)^{2}}
+	\|\partial_tU_i(s)\|^2_{L^2}+\|\nabla U_i(s)\|^2_{L^2}\right).
\end{equation}
  Summing-up the estimate  \eqref{5mai2} in $i$,  and using the fact $U_i=\partial_i u$, we write
\begin{equation} \label{5mai3}
\frac{\|\nabla u(t)\|^2_{L^2}}{
	(1+t)^{2}}+\|\nabla \partial_tu(t)\|^2_{L^2}+\|\Delta u(t)\|^2_{L^2}\leq C
 \left(\frac{1+s}{1+t}\right)^{\alpha_0}
\left( \frac{\|\nabla u(s)\|^2_{L^2}}{
	(1+s)^{2}}+
	\|\nabla \partial_tu(s)\|^{2}_{L^2}+\|\Delta u(s)\|^2_{L^2}\right),
\end{equation}
which ends the proof of the estimate \eqref{5mai1}.\\
{\bf Energy estimate \eqref{5mai5}.} If the initial data $(u_0,u_1)\in H^{3}(\mathbb{R}^n)\times H^{2}(\mathbb{R}^n)$, by letting $W=\Delta u$, we deduce
\begin{equation}\label{utt}
\begin{cases}
\displaystyle\partial_{t}^2W-\Delta W +\frac{\mu}{1+t} \partial_{t}W =0, &x\in \mathbb{R}^n,\,t>0,\\
W(x,0)=   W_0(x),\,\partial_tW(x,0)=  W_1(x),&x\in \mathbb{R}^n,\\
\end{cases}
\end{equation}
where $W_0(x)=   \Delta u_0(x),$ and $W_1(x)=  \Delta u_1(x).$
Similarly, by applying \eqref{0A1} to the mild solution $W(t)=\Delta u(t)$ of \eqref{utt}, we conclude, for  all  
  $t\ge s\ge0$, that
\begin{equation} \label{6mai1}
\frac{\|\Delta u(t)\|^2_{L^2}}{
	(1+t)^{2}}
+	\|\partial_t\Delta u(t)\|^2_{L^2}+\|\nabla \Delta u(t)\|^2_{L^2}\leq C
 \left(\frac{1+s}{1+t}\right)^{\alpha_0}
\left( \frac{\|\Delta u(s)\|^2_{L^2}}{
	(1+s)^{2}}
+	\|\partial_t\Delta u(s)\|^2_{L^2}+\|\nabla \Delta u(s)\|^2_{L^2}\right).
\end{equation}
Clearly, we deduce that estimate \eqref{5mai5} holds.
 This concludes the proof of Proposition \ref{prop2.1}.
\end{proof}
\begin{remark}\label{rmk2.1}
	Similarly to the existence part in the proof of Proposition \ref{prop2.1} and using Lemmas \ref{LE}, \ref{LEbis}, \ref{LEbbis}, the operator $R(t,s)$ and $S(t,s)$ can be
	extended uniquely such that 
	\begin{align}\label{Extend.Operator}
	R(t,s):\,H^1(\mathbb{R}^n)\times L^2(\mathbb{R}^n)\,\longrightarrow\, X_s
	\end{align}
	where
	$$X_s:=\mathcal{C}\left([s,\infty),H^1(\mathbb{R}^n)\right)\cap
	\mathcal{C}^1\left([s,\infty),L^2(\mathbb{R}^n)\right).$$
\end{remark}

\section{Main Result}\label{sec-main}\label{sec3}

The aim of this  section is  to state our main result for which  we will write  the definition of mild solution of the system (\ref{NLW}). Recall that $f(\partial_tu)=|\partial_tu|^p $ or $|\partial_tu|^{p-1}\partial_tu$.
\begin{definition}(Mild solution)\\ 
Let $1\leq n\leq 3$. Assume that 
$$(u_0,u_1)\in H^{m+2}(\mathbb{R}^n)\times H^{m+1}(\mathbb{R}^n),\qquad\hbox{with} \quad m=\left\{\begin{array}{ll}0&\quad \hbox{when}\,\,n=1,\\\\
  1&\quad \hbox{when}\,\,n=2,3.\\
  \end{array}
  \right.
  $$ We say that a function $u$ is a
mild solution of (\ref{NLW}) if
$$u\in \mathcal{C}^1([0,T],H^{m+1}(\mathbb{R}^n)),$$
and $u$ has the initial data $u(0)=u_0$, $u_t(0)=u_1$, and satisfies
the integral equation
 \begin{equation}\label{mildsolution}
u(t,x)=R(t)(u_0,u_1)+\int_0^tS(t,s)f(\partial_tu)\,ds,
\end{equation} 
in the sense of $H^{m+1}(\mathbb{R}^n)$.
\end{definition}

If $T<\infty$,
then $u$ is called a local-in-time mild  solution, and if $T=\infty$, then $u$ is called a global-in-time  mild
solution or global mild solution.

In the following, we state the  main results of this paper.
\begin{theorem}$(\hbox{Global existence: Case of $n=1$})$\label{globalexistence1}\\ 
Assume that  $u_0 \in H^2(\R)$ and $u_1 \in H^1(\R)$.
If $$p>1+\frac{2}{\alpha},$$
 where $\alpha= \min (2,\mu)$, there exists a positive constant $0<\varepsilon_0\ll1$, small enough, such that for any initial data satisfying 
$$\|u_0\|_{H^2}+\|u_1\|_{H^1}\leq\varepsilon_0,$$
 there is a uniquely global mild solution 
$$u\in  \mathcal{C}([0,\infty),H^2(\mathbb{R}))\cap   \mathcal{C}^1([0,\infty),H^1(\mathbb{R}))\cap \mathcal{C}^2([0,\infty),L^2(\mathbb{R})).$$
Moreover, the solution satisfies the following estimates
$$(1+t)^{-1}\|u(t)\|_{L^2}+
\|\nabla u(t)\|_{H^1}+\| \partial_t u(t)\|_{H^1}\leq C\,(1+t)^{-\frac{\alpha_0}{2}},$$
where $\alpha_0$ is given in \eqref{alpha0}.
\end{theorem}

\begin{remark}
We emphasize once again that
$$p_G(1,\mu)= 1+\frac{2}{\mu}$$
is the predicted critical exponent in the one-dimensional case when $\mu \le 2$. We conjecture that this observation may also hold for equations similar to \eqref{NLW}, including those with Tricomi-type or mass terms, as explored in \cite{Our8, Our5}, or for coupled wave systems studied in \cite{Our4, Our9}, where certain blow-up results have been established.
\end{remark}

\begin{theorem}$(\hbox{Global  existence: Case of $n=2,3$})$\label{globalexistence2}\\ 
Let $n=2,3$.  Assume that  $u_0 \in H^3(\R^n)$ and $u_1 \in H^2(\R^n)$.  If 
\begin{equation}\label{O}\begin{cases}
p>1+\frac{2}{\mu},&\,\,\hbox{if}\,\,\mu\leq 1,\\\\
p\geq 3,&\,\,\hbox{if}\,\, \mu>1,
\end{cases}
\end{equation}
then there exists a positive constant $0<\varepsilon_0\ll1$, small enough, such that for any initial data satisfying 
$$\|u_0\|_{H^3}+\|u_1\|_{H^2}\leq\varepsilon_0,
$$
 there is a uniquely global mild solution 
$$u\in  \mathcal{C}([0,\infty),H^3(\mathbb{R}^n))\cap   \mathcal{C}^1([0,\infty),H^2(\mathbb{R}^n)).
$$
Moreover, the solution satisfies the following estimates
$$(1+t)^{-1}\|u(t)\|_{L^2}+\|\nabla u(t)\|_{H^2}+
\|\partial_t u(t)\|_{H^2}\leq C\,(1+t)^{-\frac{\alpha_0}{2}},$$
where $\alpha_0$ is given in \eqref{alpha0}.
\end{theorem}
%

\begin{remark}\label{rem1} Let $n=2,3$. Note that (\ref{O}) implies $p>1+\frac{2}{\alpha}$ and $p\geq3$ where $\alpha=\min\{2,\mu\}$. Indeed,\\
\begin{itemize}
\item Case of $\mu\leq 1$. In this case, $\alpha=\mu$. We guarantee that
$$p>1+\frac{2}{\mu}=1+\frac{2}{\alpha}\geq3.$$
\item Case of $1<\mu\leq 2$. In this case, $\alpha=\mu$. We guarantee that
$$p\geq3>1+\frac{2}{\mu}=1+\frac{2}{\alpha}.$$
\item Case of $\mu> 2$. In this case, $\alpha=2$. We guarantee that
$$p\geq3>2=1+\frac{2}{\alpha}.$$
\end{itemize}
\end{remark}

\begin{remark} Let $n=2,3$, and
$$p_G(n+\mu)=1+\frac{2}{n+\mu-1}.$$
As $p\geq 3>p_G(n+\mu)$, it is still an open problem to see whether the solution blows-up or exists globally in time when 
$$
\begin{cases}
p_G(n+\mu)<p\leq1+ \frac{2}{\mu},&\,\,\hbox{if}\,\,\mu\leq 1,\\\\
p_G(n+\mu)<p< 3,&\,\,\hbox{if}\,\, \mu>1.
\end{cases}
$$
\end{remark}


\section{Proof of Theorems  \ref{globalexistence1}, \ref{globalexistence2}}	\label{sec4}

In this Section  is devoted to the prove Theorems \ref{globalexistence1} and \ref{globalexistence2}.
To handle the nonlinear term in the proofs, we first recall the following  elementary inequalities and lemmas
 that are useful in our proofs
 \begin{lemma}\label{basic}
Let $p\geq1, a,b\in  \R$. Then 
\begin{eqnarray}
||a|^p-|b|^p|&\leq &
C (|a|^{p-1}+|b|^{p-1}) |a-b|, \label{ab}\\
||a|^{p-1}a-|b|^{p-1}b|&\leq &
C (|a|^{p-1}+|b|^{p-1})|a-b|. \label{abb}
\end{eqnarray}
\end{lemma}

\begin{lemma}\label{handle}
Let $p\geq2$ and $F,G:\mathbb{R}\rightarrow\mathbb{R}$ such that $F(u(x))=|u(x)|^p$, and $G(u(x))=|u(x)|^{p-1}u(x)$, for all $u:\mathbb{R}^n\rightarrow\mathbb{R}$. Then  for all  $u,v:\mathbb{R}^n\rightarrow\mathbb{R}$, we have
\begin{equation}\label{21A}
|F(u(x))-F(v(x))|\leq
C(|u(x)|^{p-1}+|v(x)|^{p-1}) |u(x)-v(x)|,
\end{equation}
\begin{equation}\label{21AG}
|G(u(x))-G(v(x))|\leq
C(|u(x)|^{p-1}+|v(x)|^{p-1}) |u(x)-v(x)|,
\end{equation}
\begin{equation}\label{21B}
\big|\nabla \big[F(u(x))\big]-\nabla \big[F(v(x))\big]\big|\leq C \ |u(x)|^{p-1}|\nabla (u(x)-v(x))|+C |\nabla v(x) |
(|u(x)|^{p-2}+|v(x)|^{p-2}) |u(x)-v(x)|,
\end{equation}
for all $x\in\mathbb{R}^n$, and 
\begin{equation}\label{21BG}
\big|\nabla \big[G(u(x))\big]-\nabla \big[G(v(x))\big]\big|\leq C \ |u(x)|^{p-1}|\nabla (u(x)-v(x))|+C |\nabla v(x) |
(|u(x)|^{p-2}+|v(x)|^{p-2}) |u(x)-v(x)|,
\end{equation}
a.e. $x\in\mathbb{R}^n$. In addition, if $p\geq 3$ then 
\begin{eqnarray}\label{21C}
\big|\Delta \big[F(u(x))\big]-\Delta \big[F(v(x))\big]\big|&\leq& C \ |u(x)|^{p-1}|\Delta (u(x)-v(x))|\nonumber\\
&&+\,C\, |\Delta v(x) |\ 
\big(  |u(x)|^{p-2}+|v(x)|^{p-2}\big) |u(x)-v(x)|\nonumber
\\
&&+\,C\,|u(x)|^{p-2}\big(|\nabla u(x)|+ |\nabla v(x)|\big) |\nabla (u(x)-v(x))|\nonumber\\
&&+\,C\, |\nabla v(x) |^2\ 
\big(  |u(x)|^{p-3}+|v(x)|^{p-3}\big) |u(x)-v(x)|,
\end{eqnarray}
and
\begin{eqnarray}\label{21CG}
\big|\Delta \big[G(u(x))\big]-\Delta \big[G(v(x))\big]\big|&\leq& C \ |u(x)|^{p-1}|\Delta (u(x)-v(x))|\nonumber\\
&&+\,C\, |\Delta v(x) |\ 
\big(  |u(x)|^{p-2}+|v(x)|^{p-2}\big) |u(x)-v(x)|\nonumber
\\
&&+\,C\,|u(x)|^{p-2}\big(|\nabla u(x)|+ |\nabla v(x)|\big) |\nabla (u(x)-v(x))|\nonumber\\
&&+\,C\, |\nabla v(x) |^2\ 
\big(  |u(x)|^{p-3}+|v(x)|^{p-3}\big) |u(x)-v(x)|,
\end{eqnarray}
a.e. $x\in\mathbb{R}^n$.
\end{lemma}

\begin{proof}
Note that \eqref{21A} follows from \eqref{ab}. A straightforward calculation implies
\begin{eqnarray}\label{20A}
\partial_iF(u(x))-\partial_iF(v(x))&=&p\Big(|u(x)|^{p-2}u(x)\partial_iu(x)-|v(x)|^{p-2}v(x)\partial_iv(x)\Big)\label{20}\\
&=&p|u(x)|^{p-2}u(x)(\partial_iu(x)-\partial_iv(x))+p
\partial_iv(x) (|u(x)|^{p-2}u(x)-|v(x)|^{p-2}v(x)),\nonumber
\end{eqnarray}
for all $x\in\mathbb{R}^n$. Consequently, we derive
\begin{eqnarray}\label{21d}
\big|\nabla \big[F(u)\big]-\nabla \big[F(v)\big]\big|
&\leq& C \ |u|^{p-1}|\nabla u-\nabla v|+
C |\nabla v |\ 
\big|  |u|^{p-2}u-|v|^{p-2}v\big|,
\end{eqnarray}
for all $x\in\mathbb{R}^n$. By taking into account the inequality 
\eqref{abb} together with $p\geq2$, and \eqref{21d} we conclude \eqref{21B}. In
order to derive estimate \eqref{21C}, we first 
differentiate the identity \eqref{20} 
\begin{eqnarray}\label{20E}
\partial^2_iF(u)-\partial^2_iF(v)&=&p\Big(|u|^{p-2}u\partial^2_iu-|v|^{p-2}v\partial^2_iv\Big)+p(p-1)\big(|u|^{p-2}(\partial_iu)^2-|v|^{p-2}(\partial_iv)^2\big)\nonumber\\\nonumber\\
&=&p|u|^{p-2}u(\partial^2_iu-\partial^2_iv)+p
\partial^2_iv (|u|^{p-2}u-|v|^{p-2}v)\nonumber\\
&&+p(p-1)|u|^{p-2}\big( \partial_iu-\partial_iv\big)\big( \partial_iu+\partial_iv\big)\\
&&+p(p-1)(\partial_iv)^2 
\Big(|u|^{p-2}-|v|^{p-2}\big).\nonumber
\end{eqnarray}
a.e. $x\in\mathbb{R}^n$, and then we use the fact that $p\geq 3$, \eqref{ab}, and  \eqref{abb}. Similarly, we get \eqref{21AG},\eqref{21BG}, and \eqref{21CG}.
This ends the proof of Lemma \ref{handle}.
\end{proof}

\subsection{Proof of Theorem \ref{globalexistence1}}
Assume that  $u_0 \in H^2(\R)$ and $u_1 \in H^1(\R)$. Let $p>1+\frac{2}{\alpha},$
 where $\alpha= \min (2,\mu)$. Let  $0<\varepsilon_0\ll1$, small enough, such that  
$\|u_0\|_{H^2}+\|u_1\|_{H^1}\leq\varepsilon_0$.
 We start by introducing, for $T>0$, the space of energy solutions
$$X(T)=  \mathcal{C}([0,T],H^2(\mathbb{R}))\cap \mathcal{C}^1([0,T],H^1(\mathbb{R})),$$
equipped with the norm
$$\|v\|_{X(T)}=\sup_{0\leq t\leq T}(1+t)^{\frac{\alpha_0}{2}}\left\{ (1+t)^{-1}\| v(t)\|_{L^2}+
\|\nabla v(t)\|_{H^1}+\|\partial_t v(t)\|_{H^1}\right\},$$
 for any $v\in X(T)$, where $\alpha_0$ is given in \eqref{alpha0}. We are going to use the Banach fixed-point theorem. Let us define the following complete metric space $B_{M}(T)=\{v\in X(T);\,\,\|v\|_{X(T)}\leq  M\},$ where $M>0$ is a positive constant that will be chosen later. By Proposition \ref{prop2.1}, and estimates (\ref{An1}) and (\ref{Bn1}) below, we may define a mapping $\Phi:B_M(T)\rightarrow X(T)$ such that
$$\Phi(u)(t)=R(t)(u_0,u_1)+\int_0^tS(t,s)f(\partial_tu)\,ds=:\Phi(u)^{lin}(t)+\Phi(u)^{nl}(t),\quad\hbox{for $u\in B_M(T)$}.$$
At this stage, we divide our proof into 3 steps .\\

\noindent {\bf Step 1.} We prove in this step that $\Phi:B_M(T)\longrightarrow B_M(T)$. Let $u\in B_M(T)$.\\

\noindent $\bullet$  Estimation of $(1+t)^{-1} \|\Phi(u)(t)\|_{L^2}+ \|\partial_t\Phi(u)(t)\|_{L^2}+\|\nabla\Phi(u)(t)\|_{L^2}$. Using  \eqref{0A1} in the particular case where $s=0$, we have
\begin{equation} \label{GG1}
(1+t)^{-1}
\|\Phi(u)^{lin}(t)\|_{L^2}+
\|\nabla \Phi(u)^{lin}(t)\|_{L^2}+
\|\partial_t \Phi(u)^{lin}(t)\|_{L^2}
	\leq C \,\varepsilon_0
 ({1+t})^{-\frac{\alpha_0}2}.
\end{equation}
Moreover, 
 by using  \eqref{0A1}  again in the case where the initial data is $(0, f(\partial_tu(s)))$, we obtain
\begin{equation}\label{T0}
(1+t)^{-1}
\| \Phi(u)^{nl}(t)\|_{L^2}+
\|\partial_t \Phi(u)^{nl}(t)\|_{L^2}+\|\nabla  \Phi(u)^{nl}(t)\|_{L^2}\leq C\,\int_0^t\left(\frac{1+s}{1+t}\right)^{\frac{\alpha_0}{2}}\|f(\partial_tu(s))\|_{L^2}\,ds.
\end{equation}
By the Sobolev embedding $H^1(\mathbb{R})\hookrightarrow L^\infty(\mathbb{R})$,  and the fact that $u\in B_M(T)$, we infer 
\begin{align}\label{S01}
\|\partial_tu(s)\|_{L^{\infty}}\leqslant C \|\partial_tu(s)\|_{H^1}\leqslant  C(1+s)^{-\frac{\alpha_0}{2}}\|u\|_{X(T)}\le C M(1+s)^{-\frac{\alpha_0}{2}}.
\end{align}
Therefore, by \eqref{S01} and  exploiting again  the fact that $u\in B_M(T)$, we conclude
\begin{equation}\label{An1}
\|f(\partial_t u(s))\|_{L^2}\leq \|\partial_tu(s)\|_{\infty}^{p-1}\|\partial_tu(s)\|_{L^2}\leq 
C  M^p(1+s)^{-\frac{\alpha_0 p}2}.
\end{equation}
Combining \eqref{T0} and \eqref{An1}, we obtain
\begin{equation}\label{Ty}
(1+t)^{-1}\| \Phi(u)^{nl}(t)\|_{L^2}+
\|\partial_t \Phi(u)^{nl}(t)\|_{L^2}+\|\nabla  \Phi(u)^{nl}(t)\|_{L^2}\leq C\,M^p(1+t)^{-\frac{\alpha_0}{2}}\int_0^t(1+s)^{-\frac{(p-1)\alpha_0}{2}}\,ds.
\end{equation}
Since $p>1+\frac{2}{\alpha}$, one can choose in the case of $\mu=2$, $\eta$ sufficiently small (namely $0<\eta<\frac{2(p-2)}{p-1}$), we get 
\begin{equation}\label{alpha-alpha0}
p>1+\frac{2}{\alpha_0}\quad\hbox{which guarantees}\quad \frac{(p-1)\alpha_0}{2}>1.
\end{equation}
 Now, using \eqref{alpha-alpha0} in \eqref{Ty}, we deduce 
\begin{equation}\label{G200}
(1+t)^{-1}\| \Phi(u)^{nl}(t)\|_{L^2}+
\|\partial_t \Phi(u)^{nl}(t)\|_{L^2}+\|\nabla  \Phi(u)^{nl}(t)\|_{L^2}\leq
 C\,M^p(1+t)^{-\frac{\alpha_0}{2}}.
\end{equation}
Hence, by \eqref{GG1}, and \eqref{G200}, we get
\begin{equation}\label{G2T}
(1+t)^{-1}\| \Phi(u)(t)\|_{L^2}+
\|\partial_t \Phi(u)(t)\|_{L^2}+\|\nabla  \Phi(u)(t)\|_{L^2}\leq  C \,\varepsilon_0
 ({1+t})^{-\frac{\alpha_0}2}+\,
 C\,M^p(1+t)^{-\frac{\alpha_0}{2}}.
\end{equation}

\noindent $\bullet$ Estimation of $\|\nabla\partial_t\Phi(u)(t)\|_{L^2}+\|\Delta \Phi(u)(t)\|_{L^2}$. Using \eqref{5mai1} in the particular case where $s=0$, we have
\begin{equation}\label{H6} 
\|\nabla \partial_t \Phi(u)^{lin}(t)\|_{L^2}+\|\Delta \Phi(u)^{lin}(t)\|_{L^2}
	\leq C \,\varepsilon_0
 ({1+t})^{-\frac{\alpha_0}2}.
\end{equation}
Moreover, 
 by using  \eqref{5mai1} again in the case where the initial data is $(0, f(\partial_tu(s)))$, we obtain
\begin{equation}\label{fh}
\|\nabla \partial_t \Phi(u)^{nl}(t)\|_{L^2}+\|\Delta \Phi(u)^{nl}(t)\|_{L^2}
\leq C\,\int_0^t\left(\frac{1+s}{1+t}\right)^{\frac{\alpha_0}{2}}\|\nabla f(\partial_tu(s))\|_{L^2}\,ds.
\end{equation}
By \eqref{S01} and exploiting again  the fact that $u\in B_M(T)$, we have
\begin{equation}\label{Bn1}
\|\nabla f(\partial_tu(s))\|_{L^2}=\|p|\partial_tu|^{p-1}\nabla \partial_tu\|_{L^2}\leq p\|\partial_tu\|_\infty^{p-1}\|\nabla \partial_tu\|_{L^2}\leq C  M^p(1+s)^{-\frac{p\alpha_0}{2}}.
\end{equation}
Therefore, inserting \eqref{Bn1} in  \eqref{fh}, we conclude
\begin{equation} 
\|\nabla \partial_t \Phi(u)^{nl}(t)\|_{L^2}+\|\Delta \Phi(u)^{nl}(t)\|_{L^2}
\leq C\,M^p(1+t)^{-\frac{\alpha_0}{2}}\int_0^t(1+s)^{-\frac{(p-1)\alpha_0}{2}}\,ds.
\end{equation}
By \eqref{alpha-alpha0}, we deduce that
\begin{equation} \label{7mai3}
\|\nabla \partial_t \Phi(u)^{nl}(t)\|_{L^2}+\|\Delta \Phi(u)^{nl}(t)\|_{L^2}
\leq C\,M^p(1+t)^{-\frac{\alpha_0}{2}}.
\end{equation}
Therefore, by \eqref{H6}, and \eqref{7mai3}, we get
\begin{equation} \label{G2000}
\|\nabla \partial_t \Phi(u)(t)\|_{L^2}+\|\Delta \Phi(u)(t)\|_{L^2}
\leq C\,\varepsilon_0 (1+t)^{-\frac{\alpha_0}{2}}
+ C\,M^p(1+t)^{-\frac{\alpha_0}{2}}.
\end{equation}
Summing up the  estimates \eqref{G2T} and \eqref{G2000}, we conclude that
$$ \|\Phi(u)\|_{X(T)}\leq  C\,\varepsilon_0+C\,M^p.$$
By choosing $M>0$ such that $C\,M^{p-1}\leq 1/2$, and then $0<\varepsilon_0\ll1$ such that $C\,\varepsilon_0 \leq M/2$, we arrive at
$$ \|\Phi(u)\|_{X(T)}\leq M,$$
i.e. $\Phi(u)\in B_M(T)$.\\

\noindent {\bf Step 2.} We show that $\Phi$ is a contraction in $B_M(T)$. Let $u,v\in B_M(T)$.\\

\noindent $\bullet$ Estimation of $(1+t)^{-1}\|\Phi(u)(t)-\Phi(v)(t)\|_{L^2}+\|\partial_t(\Phi(u)-\Phi(v))(t)\|_{L^2}+\|\nabla(\Phi(u)-\Phi(v))(t)\|_{L^2}$.

\noindent By using  the estimate  \eqref{0A1}  in the case where the initial data is $(0, f(\partial_tu(s))-f(\partial_tv(s)))$, we obtain
\begin{multline}\label{TR}
(1+t)^{-1}
\| \Phi(u)(t)-\Phi(v)(t)
\|_{L^2}+
\|\partial_t \Phi(u)(t)-\partial_t \Phi(v)(t)\|_{L^2}+\|\nabla  \Phi(u)(t)-\nabla  \Phi(v)(t)\|_{L^2}\\
\leq C\,\int_0^t\left(\frac{1+s}{1+t}\right)^{\frac{\alpha_0}{2}}\|f(\partial_tu(s))-f(\partial_tv(s))\|_{L^2}\,ds.
\end{multline}
According to the basic inequality \eqref{21A}-\eqref{21AG}, we get
$$\|f(\partial_tu(s))-f(\partial_tv(s))\|_{L^2}\leq \|\partial_tu(s)-\partial_tv(s)\|_{L^2}\left(\|\partial_tu(s)\|^{p-1}_{\infty}+\|\partial_tv(s)\|^{p-1}_{\infty}\right).$$
We notice that 
$$\|\partial_tu(s)-\partial_tv(s)\|_{L^2}\leq (1+s)^{-\frac{\alpha_0}{2}}\|u-v\|_{X(T)}. $$
In addition, by applying \eqref{S01} to $u,v\in B_M(T)$, we get
$$
\|\partial_tu(s)\|_{\infty}^{p-1},\|\partial_tv(s)\|_{\infty}^{p-1}\leq C M^{p-1}(1+s)^{-\frac{\alpha_0(p-1)}{2}}.
$$
The above estimates imply that
\begin{equation}\label{Fn1}
\|f(\partial_tu(s))-f(\partial_tv(s))\|_{L^2}\leq CM^{p-1}(1+s)^{-\frac{p\alpha_0}{2}}\|u-v\|_{X(T)}. 
\end{equation}
Hence, by using \eqref{TR} and \eqref{Fn1} we infer
\begin{multline}\label{TRR}
(1+t)^{-1}
\| \Phi(u)(t)-\Phi(v)(t)
\|_{L^2}+
\|\partial_t \Phi(u)(t)-\partial_t \Phi(v)(t)\|_{L^2}+\|\nabla  \Phi(u)(t)-\nabla  \Phi(v)(t)\|_{L^2}\\
\leq C\,M^{p-1}(1+t)^{-\frac{\alpha_0}{2}} \|u-v\|_{X(T)}
\int_0^t(1+s)^{-\frac{(p-1)\alpha_0}{2}}\,ds.
\end{multline}
Making use of \eqref{alpha-alpha0}  with \eqref{TRR}, we derive
\begin{multline}\label{contractionA}
(1+t)^{-1}
\| \Phi(u)(t)-\Phi(v)(t)
\|_{L^2}+
\|\partial_t \Phi(u)(t)-\partial_t \Phi(v)(t)\|_{L^2}+\|\nabla  \Phi(u)(t)-\nabla  \Phi(v)(t)\|_{L^2}\\
\leq C\,M^{p-1}(1+t)^{-\frac{\alpha_0}{2}}\|u-v\|_{X(T)}.
\end{multline}
\noindent $\bullet$ Estimation of $\|\Delta(\Phi(u)(t)-\Phi(v)(t))\|_{L^2}$+$\|\nabla\partial_t(\Phi(u)(t)-\Phi(v)(t))\|_{L^2}$. By  using \eqref{5mai1},  we have
\begin{eqnarray}
&{}& \|\Delta\Phi(u)(t)-\Delta\Phi(v)(t)\|_{L^2}+\|\nabla\partial_t \Phi(u)(t)-\nabla\partial_t\Phi(v)(t)\|_{L^2}\nonumber\\
&{}&\leq\, C\,\int_0^t\left(\frac{1+s}{1+t}\right)^{\frac{\alpha_0}{2}}\|\nabla \left(f(\partial_tu(s))-f(\partial_tv(s))\right)\|_{L^2}\,ds.\label{Hn1}
\end{eqnarray}
Since $\alpha \le 2$, then  we have $p>1+\frac{2}{\alpha}\ge 2$, which allows us to use \eqref{21B}-\eqref{21BG}. Therefore
\begin{align*}
&\left\|\nabla \left(f(\partial_tu(s))-f(\partial_tv(s))\right)\right\|_{L^2}\\
&\leqslant C\|\nabla \partial_tw(s)\|_{L^{2}}\|\partial_tu(s)\|^{p-1}_{L^{\infty}}+C\|\partial_tw(s)\|_{L^{\infty}}\left(\| \partial_tu(s)\|^{p-2}_{L^{\infty}}+\| \partial_tv(s)\|^{p-2}_{L^{\infty}}\right)\|\nabla \partial_tv(s)\|_{L^{2}},
\end{align*}
where $w(t,x)\doteq u(t,x)-v(t,x)$. We notice from the Sobolev embedding $H^1(\mathbb{R})\hookrightarrow L^\infty(\mathbb{R})$ that
\begin{align*}
\|\partial_tu(s)\|_{L^{\infty}}&\leqslant C\|\partial_tu(s)\|_{H^1}\leqslant  C(1+s)^{-\frac{\alpha_0}{2}}\|u\|_{X(T)}\leqslant  CM(1+s)^{-\frac{\alpha_0}{2}},\\
\|\partial_tv(s)\|_{L^{\infty}}&\leqslant C\|\partial_tv(s)\|_{H^1}\leqslant C(1+s)^{-\frac{\alpha_0}{2}}\|v\|_{X(T)} \leqslant CM(1+s)^{-\frac{\alpha_0}{2}},\\
\|\partial_tw(s)\|_{L^{\infty}}&\leqslant C\|\partial_tw(s)\|_{H^1}\leqslant C(1+s)^{-\frac{\alpha_0}{2}}\|w\|_{X(T)}=C(1+s)^{-\frac{\alpha_0}{2}}\|u-v\|_{X(T)},\\
\|\nabla \partial_tv(s)\|_{L^{2}}&\leqslant (1+s)^{-\frac{\alpha_0}{2}}\|v\|_{X(T)}  \le M(1+s)^{-\frac{\alpha_0}{2}},\nonumber\\
\|\nabla \partial_tw(s)\|_{L^{2}}&\leqslant (1+s)^{-\frac{\alpha_0}{2}}\|w\|_{X(T)}=(1+s)^{-\frac{\alpha_0}{2}}\|u-v\|_{X(T)}.\nonumber
\end{align*}
Therefore,
\begin{equation}\label{Gn1}\left\|\nabla \left(f(\partial_tu(s))-f(\partial_tv(s))\right)\right\|_{L^2}\leqslant CM^{p-1}(1+s)^{-\frac{p\alpha_0}{2}}\|u-v\|_{X(T)}
\end{equation}
By plugging  (\ref{Gn1}) into (\ref{Hn1}),  and using   \eqref{alpha-alpha0}, we get
\begin{equation}\label{contractionB}
\|\Delta\Phi(u)(t)-\Delta\Phi(v)(t)\|_{L^2}+\|\nabla\partial_t \Phi(u)(t)-\nabla\partial_t\Phi(v)(t)\|_{L^2}\leq CM^{p-1}(1+t)^{-\frac{\alpha_0}{2}}\|u-v\|_{X(T) }.
\end{equation}
Summing up the above estimates \eqref{contractionA} and \eqref{contractionB}, we conclude that
$$ \|\Phi(u)-\Phi(v)\|_{X(T)}\leq  C\,M^{p-1}\|u-v\|_{X(T)}.$$
By choosing $M>0$ such that $C\,M^{p-1}\leq 1/2$, we arrive at
$$ \|\Phi(u)-\Phi(v)\|_{X(T)}\leq \frac{1}{2}\|u-v\|_{X(T)}.$$
\noindent {\bf Step 3.}  By the Banach fixed point theorem, there exists a unique mild solution $u\in X(T)$ to problem (\ref{NLW}). This completes the proof of Theorem \ref{globalexistence1}.\\

\subsection{Proof of Theorem \ref{globalexistence2}}
 We start by introducing, for $T>0$, the space of energy solutions
$$X(T)= \mathcal{C}([0,T],H^3(\mathbb{R}^n))\cap \mathcal{C}^1([0,T],H^2(\mathbb{R}^n))$$
equipped with the norm
$$\|v\|_{X(T)}=\sup_{0\leq t\leq T}(1+t)^{\frac{\alpha_0}{2}}\left\{(1+t)^{-1}\| v(t)\|_{L^2}+\|\nabla v(t)\|_{H^2}+\|\partial_t v(t)\|_{H^2}\right\},$$
 for any $v\in X(T)$, where $\alpha_0$ is given in \eqref{alpha0}. Let us define the following complete metric space $B_{M}(T)=\{v\in X(T);\,\,\|v\|_{X(T)}\leq  M\},$ where $M>0$ is a positive constant that will be chosen later. By Proposition \ref{prop2.1}, and estimates (\ref{A}), (\ref{B001}), and (\ref{K3}) below, we define a mapping $\Phi:B_M(T)\rightarrow X(T)$ such that
$$\Phi(u)(t)=R(t)(u_0,u_1)+\int_0^tS(t,s)f(\partial_tu)\,ds=:\Phi(u)^{lin}(t)+\Phi(u)^{nl}(t),\quad\hbox{for $u\in B_M(T)$}.$$
In order to apply the Banach fixed-point theorem, we divide our proof into 3 steps.\\

\noindent {\bf Step 1.} We start by proving that $\Phi:B_M(T)\longrightarrow B_M(T)$. Let $u\in B_M(T)$.\\

\noindent$\bullet$ Estimation of $(1+t)^{-1}\|\Phi(u)(t)\|_{L^2}+\|\partial_t\Phi(u)(t)\|_{L^2}+\|\nabla\Phi(u)(t)\|_{L^2}$. From the estimate \eqref{0A1}, we can write
\begin{equation}\label{G1}
(1+t)^{-1}\| \Phi(u)^{lin}(t)\|_{L^2}+
\|\partial_t \Phi(u)^{lin}(t)\|_{L^2}+\|\nabla \Phi(u)^{lin}(t)\|_{L^2}\leq C\,\varepsilon_0\,(1+t)^{-\frac{\alpha_0}{2}}.
\end{equation}
Moreover, by 
using \eqref{0A1}  again in the case where the initial data is $(0, f(\partial_tu(s)))$, we obtain
\begin{equation}\label{AAA}
(1+t)^{-1}\| \Phi(u)^{nl}(t)\|_{L^2}+
\|\partial_t \Phi(u)^{nl}(t)\|_{L^2}+\|\nabla  \Phi(u)^{nl}(t)\|_{L^2}\leq C\,\int_0^t\left(\frac{1+s}{1+t}\right)^{\frac{\alpha_0}{2}}\|f(\partial_tu(s))\|_{L^2}\,ds.
\end{equation}
In view of the Sobolev embedding $H^2(\mathbb{R}^n)\hookrightarrow L^\infty(\mathbb{R}^n)$, in the case $1\le n\le 3$,  and the fact that $u\in B_M(T)$, we deduce
\begin{equation}\label{A}
\|f(\partial_tu(s))\|_{L^2}\leq \|\partial_tu(s)\|_{L^{\infty}}^{p-1}\|\partial_tu(s)\|_{L^2} \leq C \|\partial_tu(s)\|_{H^2}^{p}\le CM^p(1+s)^{-\frac{\alpha_0 p}2}.
\end{equation}
Hence, using \eqref{A}, we conclude from \eqref{AAA} that
\begin{equation}\label{10s1}
(1+t)^{-1}\| \Phi(u)^{nl}(t)\|_{L^2}+
\|\partial_t \Phi(u)^{nl}(t)\|_{L^2}+\|\nabla  \Phi(u)^{nl}(t)\|_{L^2}\leq C\,M^p(1+t)^{-\frac{\alpha_0}{2}}\int_0^t(1+s)^{-\frac{(p-1)\alpha_0}{2}}\,ds.
\end{equation}
Using \eqref{O}, and choosing $0<\eta<1$ in the case of $\mu=2$, we get 
\begin{equation}\label{alpha-alpha00}
p>1+\frac{2}{\alpha_0}\quad\hbox{which guarantees}\quad \frac{(p-1)\alpha_0}{2}>1.
\end{equation}
Therefore, by exploiting  \eqref{alpha-alpha00}, 
\eqref{10s1} leads to
\begin{equation}\label{G2}
(1+t)^{-1}\| \Phi(u)^{nl}(t)\|_{L^2}+
\|\partial_t \Phi(u)^{nl}(t)\|_{L^2}+\|\nabla  \Phi(u)^{nl}(t)\|_{L^2}\leq
 C\,M^p(1+t)^{-\frac{\alpha_0}{2}}.
\end{equation}
So, by combining \eqref{G1}, and \eqref{G2}, we conclude
\begin{equation}\label{G12}
(1+t)^{-1}\| \Phi(u)(t)\|_{L^2}+
\|\partial_t \Phi(u)(t)\|_{L^2}+\|\nabla  \Phi(u)(t)\|_{L^2}\leq
 (C\,\varepsilon_0+C M^p)(1+t)^{-\frac{\alpha_0}{2}}.
\end{equation}
\noindent$\bullet$ Estimation of $\|\Delta\Phi(u)(t)\|_{L^2}+\|\nabla\partial_t\Phi(u)(t)\|_{L^2}$. Clearly, by using  \eqref{5mai1}, we have
\begin{equation}\label{G3}
\|\Delta \Phi(u)^{lin}(t)\|_{L^2}+\|\nabla\partial_t \Phi(u)^{lin}(t)\|_{L^2}\leq C\,\varepsilon_0\,(1+t)^{-\frac{\alpha_0}{2}},
\end{equation}
and 
\begin{equation}\label{G33}
\|\Delta \Phi(u)^{nl}(t)\|_{L^2}
+\|\nabla\partial_t \Phi(u)^{nl}(t)\|_{L^2}
\leq C\, (1+t)^{-\frac{\alpha_0}{2}}\int_0^t(1+s)^{\frac{\alpha_0}{2}}\|\nabla f(\partial_tu(s))\|_{L^2}\,ds.
\end{equation}
Using the fact that $H^2(\mathbb{R}^n)\hookrightarrow L^\infty(\mathbb{R}^n)$,  in the case $1\le n\le 3$, we get
\begin{equation}\label{B001}
\|\nabla f(\partial_tu(s))\|_{L^2}\leq p\|\partial_tu(s)\|_\infty^{p-1}\|\nabla \partial_tu(s)\|_{L^2}
\leq C\|\partial_tu(s)\|_{H^2}^{p}
\leq C(1+s)^{-\frac{\alpha_0 p}2}\|u\|_{X(T)}^{p}.
\end{equation}
Plugging the  inequality \eqref{B001}, and  using the fact that  $u\in B_M(T)$, into \eqref{G33}, we obtain
\begin{equation}
\|\Delta \Phi(u)^{nl}(t)\|_{L^2}
+\|\nabla\partial_t \Phi(u)^{nl}(t)\|_{L^2}\leq C\,M^p(1+t)^{-\frac{\alpha_0}{2}}\int_0^t(1+s)^{-\frac{(p-1)\alpha_0}{2}}ds.
\end{equation}
Consequently, from \eqref{alpha-alpha00}, we deduce that
\begin{equation}\label{GF2}
\|\Delta \Phi(u)^{nl}(t)\|_{L^2}
+\|\nabla\partial_t \Phi(u)^{nl}(t)\|_{L^2}\leq  CM^p(1+t)^{-\frac{\alpha_0}2}.
\end{equation}
Therefore, from \eqref{G3} and \eqref{GF2}, we deduce that
\begin{equation}\label{GG2}
\|\Delta \Phi(u)(t)\|_{L^2}
+\|\nabla\partial_t \Phi(u)(t)\|_{L^2}\leq  (C\varepsilon_0+C M^p)(1+t)^{-\frac{\alpha_0}2}.
\end{equation}

\noindent$\bullet$ Estimation of $\|\nabla\Delta \Phi(u)(t)\|_{L^2}+\|\partial_t\Delta \Phi(u)(t)\|_{L^2}$. Applying 
\eqref{5mai5},  we have
\begin{equation}\label{G5}
\|\nabla\Delta \Phi(u)^{lin}(t)\|_{L^2}+\|\partial_t \Delta \Phi(u)^{lin}(t)\|_{L^2}\leq C\,\varepsilon_0\,(1+t)^{-\frac{\alpha_0}{2}},
\end{equation}
and
\begin{equation}\label{GG5}
\|\nabla\Delta\Phi(u)^{nl}(t)\|_{L^2}+\|\partial_t\Delta \Phi(u)^{nl}(t)\|_{L^2}\leq C\, (1+t)^{-\frac{\alpha_0}{2}}\int_0^t(1+s)^{\frac{\alpha_0}2}\|\Delta f(\partial_tu(s))\|_{L^2}\,ds.
\end{equation}
Thanks to  \eqref{21C}, and \eqref{21CG}, in the particular case where $v(x)\equiv 0$, we   write
\begin{equation}\label{K1}
\|\Delta f(\partial_tu(s))\|_{L^2}
\leq C\|\partial_tu(s)\|_\infty^{p-2}\|\nabla \partial_tu(s)\|^2_{L^4}+C\|\partial_tu(s)\|_\infty^{p-1}\|\Delta \partial_tu(s)\|_{L^2}.
\end{equation}
Note that, by  exploiting the fact that $H^1(\mathbb{R}^n)\hookrightarrow L^4(\mathbb{R}^n)$ and $H^2(\mathbb{R}^n)\hookrightarrow L^\infty(\mathbb{R}^n)$,  in the case $1\le n\le 3$, 
and    $u\in B_M(T)$,    we deduce from  \eqref{K1} that
\begin{equation}\label{K3}
\|\Delta f(\partial_tu(s))\|_{L^2}
\leq C\|\partial_tu(s)\|^{p}_{H^2}\le C M^p(1+s)^{-\frac{p\alpha_0}{2}}.
\end{equation}
 Therefore, by \eqref{GG5},  \eqref{K3}, and  \eqref{alpha-alpha00}, we infer
\begin{equation}\label{K4}
\|\nabla\Delta\Phi(u)^{nl}(t)\|_{L^2}+\|\partial_t\Delta \Phi(u)^{nl}(t)\|_{L^2} \leq 
 CM^p(1+t)^{-\frac{\alpha_0}2}.
\end{equation}
Consequently, from \eqref{G5} and \eqref{K4}, we deduce that
\begin{equation}\label{GG22}
\|\nabla\Delta\Phi(u)(t)\|_{L^2}+\|\partial_t\Delta \Phi(u)(t)\|_{L^2} \leq  (C\varepsilon_0+C M^p)(1+t)^{-\frac{\alpha_0}2}.
\end{equation}

\noindent Summing up the  estimates \eqref{G12}, \eqref{GG2},  and \eqref{GG22}, we conclude that
$$ \|\Phi(u)\|_{X(T)}\leq  C\,\varepsilon_0+C\,M^p.$$
By choosing $M>0$ such that $C\,M^{p-1}\leq 1/2$, and then $0<\varepsilon_0\ll1$ such that $ C\,\varepsilon_0\leq M/2$, we arrive at
$$ \|\Phi(u)\|_{X(T)}\leq M,$$
i.e. $\Phi(u)\in B_M(T)$.\\

\noindent {\bf Step 2.} We show that $\Phi$ is a contraction. Let $u,v\in B_M(T)$.\\

\noindent$\bullet$ Estimation of $(1+t)^{-1}\|\Phi(u)(t)-\Phi(v)(t)\|_{L^2}+\|\partial_t(\Phi(u)-\Phi(v))(t)\|_{L^2}+\|\nabla(\Phi(u)-\Phi(v))(t)\|_{L^2}$.

\noindent By using  the estimate  \eqref{0A1}  in the case where the initial data is $(0, f(\partial_tu(s))-f(\partial_tv(s)))$, we obtain
\begin{multline}
(1+t)^{-1}
\| \Phi(u)(t)-\Phi(v)(t)
\|_{L^2}+
\|\partial_t \Phi(u)(t)-\partial_t \Phi(v)(t)\|_{L^2}+\|\nabla  \Phi(u)(t)-\nabla  \Phi(v)(t)\|_{L^2}\\
\leq C\,\int_0^t\left(\frac{1+s}{1+t}\right)^{\frac{\alpha_0}{2}}\|f(\partial_tu(s))-f(\partial_tv(s))\|_{L^2}\,ds.
\end{multline}
By the basic inequality \eqref{ab}-\eqref{abb}, we get
$$\|f(\partial_tu(s))-f(\partial_tv(s))\|_{L^2}\leq \|\partial_tu(s)-\partial_tv(s)\|_{L^2}\left(\|\partial_tu(s)\|^{p-1}_{\infty}+\|\partial_tv(s)\|^{p-1}_{\infty}\right).$$
As
$$
\|\partial_tu(s)-\partial_tv(s)\|_{L^2}\leq (1+s)^{-\frac{\alpha_0}{2}}\|u-v\|_{X(T)}.
$$
therefore, using the Sobolev embedding $H^2(\mathbb{R}^n)\hookrightarrow L^\infty(\mathbb{R}^n)$, we get
\begin{equation}\label{As1}
\|\partial_tu(s)\|_{\infty}\leq \|\partial_tu(s)\|_{H^2}\leq (1+s)^{-\frac{\alpha}{2}}\|u\|_{X(T)}\leq C M(1+s)^{-\frac{\alpha_0}{2}},
\end{equation}
and
\begin{equation}\label{As2}
\|\partial_tv(s)\|_{\infty}\leq \|\partial_tv(s)\|_{H^2}\leq (1+s)^{-\frac{\alpha}{2}}\|v\|_{X(T)}\leq C M(1+s)^{-\frac{\alpha_0}{2}}.
\end{equation}
Therefore
\begin{equation}\label{F}
\|f(\partial_tu(s))-f(\partial_tv(s))\|_{L^2}\leq CM^{p-1}(1+s)^{-\frac{p\alpha_0}{2}}\|u-v\|_{X(T)},
\end{equation}
and then
\begin{multline}
(1+t)^{-1}
\| \Phi(u)(t)-\Phi(v)(t)
\|_{L^2}+
\|\partial_t \Phi(u)(t)-\partial_t \Phi(v)(t)\|_{L^2}+\|\nabla  \Phi(u)(t)-\nabla  \Phi(v)(t)\|_{L^2}\\
\leq C\,M^{p-1}\|u-v\|_{X(T)}
(1+t)^{-\frac{\alpha_0}{2}}\int_0^t(1+s)^{-\frac{(p-1)\alpha_0}{2}}\,ds.
\end{multline}
By \eqref{alpha-alpha00}, we deduce that
\begin{multline}\label{contraction3}
(1+t)^{-1}
\| \Phi(u)(t)-\Phi(v)(t)
\|_{L^2}+
\|\partial_t \Phi(u)(t)-\partial_t \Phi(v)(t)\|_{L^2}+\|\nabla  \Phi(u)(t)-\nabla  \Phi(v)(t)\|_{L^2}\\
\leq C\,M^{p-1}1+t)^{-\frac{\alpha_0}{2}}\|u-v\|_{X(T)}.
\end{multline}
\noindent$\bullet$ Estimation of $\|\Delta(\Phi(u)(t)-\Phi(v)(t))\|_{L^2}$+$\|\nabla\partial_t(\Phi(u)(t)-\Phi(v)(t))\|_{L^2}$. Using \eqref{5mai1},  we have
\begin{eqnarray}
&{}& \|\Delta\Phi(u)(t)-\Delta\Phi(v)(t)\|_{L^2}+\|\nabla\partial_t \Phi(u)(t)-\nabla\partial_t\Phi(v)(t)\|_{L^2}\nonumber\\
&{}&\leq\, C\,\int_0^t\left(\frac{1+s}{1+t}\right)^{\frac{\alpha_0}{2}}\|\nabla \left(f(\partial_tu)-f(\partial_tv)\right)\|_{L^2}\,ds.\label{H}
\end{eqnarray}
By Remark \ref{rem1},  we have $p\geq 3> 2$, which allows us to use \eqref{21B}-\eqref{21BG}. Therefore
\begin{align*}
&\left\|\nabla \left(f(\partial_tu(s))-f(\partial_tv(s))\right)\right\|_{L^2}\\
&\leqslant C\|\nabla \partial_tw(s)\|_{L^{2}}\|\partial_tu(s)\|^{p-1}_{L^{\infty}}+C\|\partial_tw(s)\|_{L^{\infty}}\left(\| \partial_tu(s)\|^{p-2}_{L^{\infty}}+\| \partial_tv(s)\|^{p-2}_{L^{\infty}}\right)\|\nabla \partial_tv(s)\|_{L^{2}},
\end{align*}
where $w(t,x)\doteq u(t,x)-v(t,x)$. We notice from the Sobolev embedding $H^2(\mathbb{R}^n)\hookrightarrow L^\infty(\mathbb{R}^n)$ that
\begin{align}\label{As3}
\|\partial_tw(s)\|_{L^{\infty}}&\leqslant C\|\partial_tw(s)\|_{H^2}\leqslant C(1+s)^{-\frac{\alpha_0}{2}}\|w\|_{X(T)}=C(1+s)^{-\frac{\alpha_0}{2}}\|u-v\|_{X(T)},\\
\|\nabla \partial_tv(s)\|_{L^{2}}&\leqslant (1+s)^{-\frac{\alpha_0}{2}}\|v\|_{X(T)}  \le CM(1+s)^{-\frac{\alpha_0}{2}},\nonumber\\
\|\nabla \partial_tw(s)\|_{L^{2}}&\leqslant (1+s)^{-\frac{\alpha_0}{2}}\|w\|_{X(T)}=(1+s)^{-\frac{\alpha_0}{2}}\|u-v\|_{X(T)},\nonumber
\end{align}
therefore, using \eqref{As1}, \eqref{As2}, we conclude that
\begin{equation}\label{G}\left\|\nabla \left(f(\partial_tu(s))-f(\partial_tv(s))\right)\right\|_{L^2}\leqslant CM^{p-1}(1+s)^{-\frac{p\alpha_0}{2}}\|u-v\|_{X(T)}
\end{equation}
By plugging (\ref{G}) into (\ref{H}), and using \eqref{alpha-alpha00}, we get
\begin{equation}\label{contraction4}
\|\Delta\Phi(u)(t)-\Delta\Phi(v)(t)\|_{L^2}+\|\nabla\partial_t \Phi(u)(t)-\nabla\partial_t\Phi(v)(t)\|_{L^2}\leq CM^{p-1}(1+t)^{-\frac{\alpha_0}{2}}\|u-v\|_{X(T) }.
\end{equation}

\noindent$\bullet$ Estimation of $\|\nabla\Delta\left(\Phi(u)(t)-\Phi(v)(t)\right)\|_{L^2}+\|\Delta\partial_t(\Phi(u)(t)-\Phi(v)(t))\|_{L^2}$. Using \eqref{5mai5},  we have
\begin{eqnarray}
&{}& \|\nabla\Delta\Phi(u)(t)-\nabla\Delta\Phi(v)(t)\|_{L^2}+\|\Delta\partial_t \Phi(u)(t)-\Delta\partial_t\Phi(v)(t)\|_{L^2}\nonumber\\
&{}&\leq\, C\,\int_0^t\left(\frac{1+s}{1+t}\right)^{\frac{\alpha_0}{2}}\|\Delta \left(f(\partial_tu)-f(\partial_tv)\right)\|_{L^2}\,ds.\label{K}
\end{eqnarray}
Using Remark \ref{rem1}, one can apply \eqref{21C}-\eqref{21CG} to get
\begin{multline*}
\|\Delta\big[f(\partial_tu(s))-f(\partial_tv(s))\big]\|_{L^2}\leq \underbrace{C\|\partial_tu(s)\|_{L^\infty}^{p-1}\|\Delta[\partial_tw(s)]\|_{L^2}}_{A_1(s)} \\
+\underbrace{C\|\Delta[\partial_tv(s)]\|_{L^2}
\Big(\|\partial_tu(s)\|_{L^\infty}^{p-2}+\|\partial_tv(s)\|_{L^\infty}^{p-2}\Big)\|\partial_tw(s)\|_{L^{\infty}}}_{A_2(s)}\\
+
\underbrace{C\|\partial_tu(s)\|_{L^\infty}^{p-2}\Big(\|\nabla \partial_tu(s)\|_{L^4}+\|\nabla\partial_tv(s)]\|_{L^4}\Big)
\|\nabla\partial_tw(s)\|_{L^4}}_{A_3(s)}
\qquad  \\
+\underbrace{C\|\nabla[\partial_tv(s)]\|_{L^4}^2
\Big(\|\partial_tu(s)\|_{L^\infty}^{p-3}+\|\partial_tv(s)\|_{L^\infty}^{p-3}\Big)\|\partial_tw(s)\|_{L^{\infty}}}_{A_4(s)},\qquad \hbox{for all}\,\, s\in [0,t].
\end{multline*}
Thanks to \eqref{As1}, \eqref{As2}, \eqref{As3}, the Sobolev embedding $H^1(\mathbb{R}^n)\hookrightarrow L^4(\mathbb{R}^n)$, and the fact that  $u,v\in B_M(T)$, we conclude that
\begin{equation}\label{As4}
\|\Delta\big[f(\partial_tu(s))-f(\partial_tv(s))\big]\|_{L^2}
\leq  CM^{p-1}(1+s)^{-\frac{\alpha_0 p}{2}}\|u-v\|_{X(T)}, \qquad \hbox{for all}\,\, s\in [0,t].
 \end{equation}
By plugging (\ref{As4}) into (\ref{K}), and using \eqref{alpha-alpha00}, we get
\begin{equation}\label{contraction5}
\|\Delta\Phi(u)(t)-\Delta\Phi(v)(t)\|_{L^2}+\|\nabla\partial_t \Phi(u)(t)-\nabla\partial_t\Phi(v)(t)\|_{L^2}\leq CM^{p-1}(1+t)^{-\frac{\alpha_0}{2}}\|u-v\|_{X(T) }.
 \end{equation}
Summing up the estimates \eqref{contraction3}, \eqref{contraction4}, and \eqref{contraction5}, we conclude that
$$ \|\Phi(u)-\Phi(v)\|_{X(T)}\leq  C\,M^{p-1}\|u-v\|_{X(T)}.$$
By choosing $M>0$ such that $C\,M^{p-1}\leq 1/2$, we arrive at
$$ \|\Phi(u)-\Phi(v)\|_{X(T)}\leq \frac{1}{2}\|u-v\|_{X(T)}.$$
\noindent {\bf Step 3.} By the Banach fixed point theorem, there exists a unique mild solution $u\in X(T)$ to problem (\ref{NLW}). This completes the proof of Theorem \ref{globalexistence2}.

\section*{Acknowledgments} The authors would like to extend their heartfelt thanks to Professor Ryo Ikehata. His expertise and support greatly enhanced our work, and we truly appreciate his contribution.
\bibliographystyle{elsarticle-num}

\noindent{\bf Address}:\\
College of Engineering and Technology, American University of the Middle East, Kuwait.\\
\vspace{-7mm}
\begin{verbatim}
e-mail:  ahmad.fino@aum.edu.kw
\end{verbatim}
 Department of Basic Sciences, Deanship of Preparatory and Supporting Studies,
  Imam Abdulrahman Bin Faisal University
P.O. Box 1982 Dammam, Saudi Arabia.\\
\vspace{-7mm}
\begin{verbatim}
e-mail:  mahamza@iau.edu.sa
\end{verbatim}

\end{document}